\documentclass[11pt]{article}
\usepackage[T1]{fontenc}
\usepackage{lmodern}
\usepackage{microtype}
\usepackage[margin=1.08in]{geometry}
\usepackage{amsmath,amssymb,amsthm,mathtools}
\usepackage{xcolor}
\usepackage{hyperref}
\usepackage{aliascnt}
\usepackage[nameinlink,noabbrev]{cleveref}
\hypersetup{
  colorlinks=true,
  linkcolor=blue!45!black,
  citecolor=blue!45!black,
  urlcolor=blue!45!black,
  pdftitle={Affine Copies of Three-Point Patterns in Sets of Integers},
  pdfauthor={Samuel Korsky},
  pdfsubject={Extremal counting of affine copies of three-point integer patterns},
  pdfkeywords={additive combinatorics, affine copies, three-point patterns, linear equations, extremal counting}
}
\allowdisplaybreaks
\numberwithin{equation}{section}
\newtheorem{theorem}{Theorem}[section]
\newaliascnt{lemma}{theorem}
\newtheorem{lemma}[lemma]{Lemma}
\aliascntresetthe{lemma}
\newaliascnt{proposition}{theorem}
\newtheorem{proposition}[proposition]{Proposition}
\aliascntresetthe{proposition}
\theoremstyle{remark}
\newaliascnt{remark}{theorem}

\aliascntresetthe{remark}
\crefname{theorem}{Theorem}{theorems}
\crefname{lemma}{Lemma}{lemmas}
\crefname{proposition}{Proposition}{propositions}
\crefname{remark}{Remark}{remarks}
\newcommand{\Z}{\mathbb Z}
\newcommand{\pos}[1]{(#1)_{+}}
\newcommand{\Fq}{\mathbb Z/q\mathbb Z}
\title{Affine Copies of Three-Point Patterns in Sets of Integers}
\author{Samuel Korsky}
\date{August 29, 2026}
\begin{document}
\maketitle
\begin{abstract}
\noindent
Let $P=\{0,a,b\}$, where $0<a<b$ and $\gcd(a,b)=1$.  For a finite
set $A\subset\Z$, let $M_P^+(A)$ count the copies
$x,x+ad,x+bd\in A$ with $d>0$, and let $M_P(A)$ count the copies with
any $d\ne0$.  We prove that every such three-point pattern other than the
arithmetic progression $\{0,1,2\}$ satisfies
\[
 M_P^+(A)\le \frac{99}{400}|A|^2+O(|A|),
 \qquad
 M_P(A)\le \frac{13}{28}|A|^2+O_P(|A|).
\]
For the particular pattern $P=\{0,1,3\}$ -- the subject of a question raised
by Ganguly and recorded as Problem~24 in Green's list of open problems -- we
sharpen the bound allowing both signs of the dilation to
\[
 M_{\{0,1,3\}}(A)\le \frac{47}{122}|A|^2+O(|A|).
\]
\end{abstract}

\section{Introduction}
After translating a three-point integer pattern and dividing all of its
differences by their common divisor, it has a unique form
\[
 P=\{0,a,b\},\qquad 0<a<b,\qquad \gcd(a,b)=1.
\]
We call this the \emph{coprime form} of the pattern.  For a finite set
$A\subset\Z$, define
\begin{align}
 M_P^+(A)
 &=\#\{(x,d)\in\Z^2:d>0,\ x,x+ad,x+bd\in A\},
 \label{eq:def-MP-plus}\\
 M_P(A)
 &=\#\{(x,d)\in\Z^2:d\ne0,\ x,x+ad,x+bd\in A\}.
 \label{eq:def-MP}
\end{align}
We call $d$ the dilation.  Thus $M_P^+$ counts only positive dilations,
while $M_P$ counts both signs.

Because $\gcd(a,b)=1$, any affine image $x+\lambda P$ contained in $\Z$ has
$x,\lambda\in\Z$: the integrality of $a\lambda$ and $b\lambda$, together with
B\'ezout's identity, implies that $\lambda$ is an integer.  Thus the affine
copies in the title are exactly the copies counted in
\eqref{eq:def-MP-plus} and \eqref{eq:def-MP}.

Translation of the pattern does not change either count.  Dividing all
pairwise differences by their common gcd can only enlarge the set of allowed
integer dilations, so an upper bound proved for the coprime form also holds
before this division.

Reflecting the pattern gives
\[
 P^\vee=\{0,b-a,b\}.
\]
Reflecting the set and changing the sign of $d$ give
\begin{equation}\label{eq:reflection-relations}
 M_P^+(A)=M_{P^\vee}^+(-A),
 \qquad
 M_P(A)=M_P^+(A)+M_{P^\vee}^+(A).
\end{equation}
Write
\[
 m_P^+(n)=\max_{|A|=n}M_P^+(A),
 \qquad
 m_P(n)=\max_{|A|=n}M_P(A),
\]
and
\[
 \gamma_P^+=\limsup_{n\to\infty}\frac{m_P^+(n)}{n^2},
 \qquad
 \gamma_P=\limsup_{n\to\infty}\frac{m_P(n)}{n^2}.
\]
Equation~\eqref{eq:reflection-relations} implies
$\gamma_P^+=\gamma_{P^\vee}^+$.

\subsection*{Basic bounds and previous work}
For a positive integer $n$, write $[n]=\{1,\ldots,n\}$. When $d>0$, the point $x+ad$ lies strictly between the other two points.
This observation gives the following general bound.
\begin{proposition}[The basic order bound and the interval count]
\label{prop:benchmarks}
Let $P=\{0,a,b\}$ with $0<a<b$, let $A\subset\Z$, and put $n=|A|$.
Then
\begin{align}
 M_P^+(A)&\le \left\lfloor\frac{(n-1)^2}{4}\right\rfloor,
 \label{eq:naive-upper-plus}\\
 M_P(A)&\le 2\left\lfloor\frac{(n-1)^2}{4}\right\rfloor.
 \label{eq:naive-upper}
\end{align}
Moreover, if $r=\lfloor(n-1)/b\rfloor$, then
\begin{align}
 M_P^+([n])&=rn-\frac b2r(r+1)=\frac1{2b}n^2+O_P(n),
 \label{eq:interval-general-plus}\\
 M_P([n])&=2rn-br(r+1)=\frac1b n^2+O_P(n).
 \label{eq:interval-general}
\end{align}
\end{proposition}
\begin{proof}
Write $A=\{u_1<\cdots<u_n\}$.  Suppose that $x+ad=u_i$ in a copy with
$d>0$.  Choosing either the lower point $x$ or the upper point $x+bd$
determines $d$, and hence determines the whole copy.  There are $i-1$
choices below $u_i$ and $n-i$ choices above it, so at most
$\min\{i-1,n-i\}$ copies can have middle point $u_i$.  Summing over $i$
gives \eqref{eq:naive-upper-plus}.  Repeating the same argument for $d<0$
gives \eqref{eq:naive-upper}.

For the interval $[n]$, a fixed positive $d$ allows exactly $n-bd$ choices
of $x$, provided $1\le d\le r$.  Summing over $d$ gives the two interval
formulas.
\end{proof}
The copies of $P$ are precisely the nonconstant integer solutions of
\begin{equation}\label{eq:general-pattern-equation}
 (b-a)X-bY+aZ=0.
\end{equation}
Indeed, the condition $\gcd(a,b)=1$ implies that every nonconstant solution
has $Y-X=ad$ and $Z-X=bd$ for a unique $d\ne0$.  Aaronson proved that every
three-variable linear equation with nonzero coefficients has arbitrarily
large sets containing at least $(1/12+o(1))n^2$ solutions
\cite{Aaronson}.  Constant solutions do not affect the coefficient of
$n^2$.  By \eqref{eq:reflection-relations}, one sign of $d$ accounts for at
least half of these solutions, and reflecting the set converts that sign to
$d>0$.  Therefore
\begin{align}
 \max\left\{\frac1{2b},\frac1{24}\right\}
 &\le \gamma_P^+\le\frac14,
 \label{eq:benchmark-window-plus}\\
 \max\left\{\frac1b,\frac1{12}\right\}
 &\le \gamma_P\le\frac12.
 \label{eq:benchmark-window}
\end{align}
Aaronson also showed that $1/12$ is the best universal lower coefficient if
one allows arbitrary triples of nonzero coefficients.  That result does not
determine the sharper behavior of \eqref{eq:general-pattern-equation}, which
is unchanged when the same integer is added to $X,Y,Z$.

A complementary geometric formulation was studied by \'Abrego,
Fern\'andez-Merchant, Katz and Kolesnikov \cite{AFMKK}.  Their order method
gives the upper bound coefficient $1/4$ when only one sign of the dilation is counted.
For $\{0,1,3\}$ they also constructed sets with positive-dilation density
$3/16$.  Thus intervals do not always maximize the positive-dilation count,
and the particular gaps in the pattern still matter after the basic order
bound has been applied.

The arithmetic progression $P=\{0,1,2\}$ is exceptional.  The interval
formulas give
\[
 M_{\{0,1,2\}}^+([n])=\frac14n^2+O(n),
 \qquad
 M_{\{0,1,2\}}([n])=\frac12n^2+O(n).
\]
Our first two theorems show that no other pattern in coprime form reaches
either coefficient.

\begin{theorem}[Positive dilations for all coprime three-point patterns]
\label{thm:positive-general}
Let $P=\{0,a,b\}$ satisfy $0<a<b$ and $\gcd(a,b)=1$, and suppose
$P\ne\{0,1,2\}$.  Then every finite nonempty $A\subset\Z$ of size $n$ satisfies
\begin{equation}\label{eq:positive-general-bound}
 M_P^+(A)\le \frac{99}{400}n^2+(n-1).
\end{equation}
Equivalently, $\gamma_P^+\le99/400$.
\end{theorem}

\begin{theorem}[Both signs of the dilation for all coprime three-point patterns]
\label{thm:both-signs-general}
Let $P=\{0,a,b\}$ satisfy $0<a<b$ and $\gcd(a,b)=1$, and suppose
$P\ne\{0,1,2\}$.  Choose an odd prime $p$ dividing one of the three
differences $a$, $b$, or $b-a$; such a prime exists by
\cref{lem:odd-prime}.  Let
\[
 q=p^{v_p(\Delta)},
\]
where $\Delta$ is the chosen difference and $v_p(\Delta)$ is the exponent
of $p$ in its prime factorization.  When several choices are available, one
may select the one minimizing $q^2(q-1)$.  If $A\subset\Z$ is finite and nonempty
and $|A|=n$, then
\begin{equation}\label{eq:general-bound}
 M_P(A)
 \le \frac{13}{28}(n^2-n)+\frac{q^2(q-1)}4(n-1).
\end{equation}
In particular, $\gamma_P\le13/28$.
\end{theorem}

Among patterns in coprime form, $\{0,1,2\}$ is therefore the only one with
$\gamma_P^+=1/4$, and it is also the only one with $\gamma_P=1/2$.  For
every other such pattern,
\begin{align}
 \max\left\{\frac1{2b},\frac1{24}\right\}
 &\le \gamma_P^+\le\frac{99}{400},
 \label{eq:general-new-window-plus}\\
 \max\left\{\frac1b,\frac1{12}\right\}
 &\le \gamma_P\le\frac{13}{28}.
 \label{eq:general-new-window}
\end{align}

\subsection*{The pattern $\{0,1,3\}$}
Put
\[
 P_0=\{0,1,3\},
 \qquad
 M^+(A)=M_{P_0}^+(A),
 \qquad
 M(A)=M_{P_0}(A).
\]
Then
\begin{equation}\label{eq:main-equation-intro}
 M(A)=\#\{(x,y,z)\in A^3:2x+z=3y,\ \text{$x,y,z$ not all equal}\}.
\end{equation}
After permuting the variables, this is the equation $x+2y=3z$ considered by
Aaronson.  Ganguly asked for the largest possible number of its solutions;
Aaronson highlighted it as an unresolved case \cite{Aaronson}.  Green
records the equivalent affine-copy question as Problem~24 and gives the
expected coefficient $1/3$ when both signs of $d$ are counted \cite{Green}.

For an interval,
\begin{align}
 M^+([n])
 &=\sum_{1\le d\le(n-1)/3}(n-3d)
 =\frac16n^2+O(n),
 \label{eq:interval-count-plus}\\
 M([n])
 &=2\sum_{1\le d\le(n-1)/3}(n-3d)
 =\frac13n^2+O(n).
 \label{eq:interval-count}
\end{align}

The construction in \cite{AFMKK} raises the positive-dilation lower
coefficient to $3/16$.  Combining this with
\cref{thm:positive-general} gives
\begin{equation}\label{eq:013-window-plus}
 \frac3{16}\le\gamma_{P_0}^+\le\frac{99}{400}=0.2475.
\end{equation}

When both signs of $d$ are allowed, we prove the following sharper finite
bound.
\begin{theorem}[Both signs of the dilation for $\{0,1,3\}$]
\label{thm:47-122}
If $A\subset\Z$ is finite and nonempty and $|A|=n$, then
\begin{equation}\label{eq:47-122-bound}
 M(A)
 \le \frac{47}{122}(n^2-n)+\frac92(n-1).
\end{equation}
\end{theorem}

Thus
\begin{equation}\label{eq:013-window}
 \frac13\le\gamma_{P_0}\le\frac{47}{122}=0.385245\ldots.
\end{equation}

\subsection*{Proof outline}
The proofs use residue classes in two different ways.  For positive $d$, the
basic order argument counts possible pairs around the middle point.  If one
gap of the pattern is divisible by $q$, two points of every copy must be in
the same residue class modulo $q$.  When the set is spread among several
classes, many of the pairs counted by the basic argument cannot occur in a
copy.  When most elements lie in one class, we divide that class by $q$ and
apply the argument again to a smaller set.  The remaining copies, which use
more than one residue class, are bounded directly.

When both signs of $d$ are allowed, we repeatedly split the set into residue
classes modulo $q,q^2,q^3,\ldots$.  Every nonconstant copy has a first level
at which its three points no longer lie in the same class.  We count the copy
at that level.  The quantities that arise are simply counts of ordered pairs
that separate at one level or the next, and each ordered pair is counted only
once when all levels are added.

For $\{0,1,3\}$, the same repeated splitting is carried out modulo $3$, but
we retain the three individual class sizes instead of replacing them by a
single total.  An auxiliary nonnegative correction term keeps track of how the distribution
at one level constrains the next.  The correction terms at consecutive
levels cancel when the inequalities are summed.

The next section proves two elementary residue-class estimates and recalls
the Lev-Pinchasi three-set bound used in the later arguments. 
The later sections treat positive $d$, both signs of $d$, and
finally the sharper modulo-$3$ argument for $\{0,1,3\}$.
\section{Three Auxiliary Estimates}\label{sec:common}
\subsection{How residue classes improve the basic order bound}
Suppose that
\[
 P=\{0,a,a+g\},
\]
and choose $q\ge2$ such that $q\mid g$ and $\gcd(a,q)=1$.  For a finite
set $S\subset\Z$, let
\begin{equation}\label{eq:def-nu-q}
 \nu_q(S)=|S|-\max_{r\bmod q}|S\cap(r+q\Z)|.
\end{equation}
Thus $\nu_q(S)$ is the number of elements outside the largest residue class
modulo $q$.  Define
\begin{equation}\label{eq:def-G}
 G(s)=
 \begin{cases}
 s^2/4,&0\le s\le1/2,\\[1mm]
 \bigl(1+(2s-1)^2\bigr)/16,&1/2\le s\le1.
 \end{cases}
\end{equation}
\begin{lemma}[Improvement when the set uses several residue classes]
\label{lem:residue-spread}
Under these assumptions, every nonempty finite $S\subset\Z$ of size $N$
satisfies
\begin{equation}\label{eq:residue-spread}
 M_P^+(S)
 \le \frac14N^2-G\left(\frac{\nu_q(S)}N\right)N^2+1.
\end{equation}
\end{lemma}
\begin{proof}
First suppose that $N=2m$.  Let $L$ be the lower half of $S$ and $U$ the
upper half, so $|L|=|U|=m$.  A copy with $d>0$ has the form $x<y<z$.
If $y\in L$, associate the copy with the pair $(x,y)$; if $y\in U$,
associate it with $(y,z)$.  Each such pair determines at most one copy.
There are
\[
 2\binom m2=m(m-1)
\]
pairs available before the congruence condition is used.  We now count
pairs that cannot belong to a copy.

Choose a residue $r_0\pmod q$ that occurs most often in $U$, and set
\[
 \sigma=|U\setminus(r_0+q\Z)|,
 \qquad
 \ell=|L\setminus(r_0+q\Z)|,
 \qquad
 t=\sigma+\ell.
\]
Consider first the pairs $(y,z)$ with $y,z\in U$.  In every copy,
\[
 z-y=gd
\]
is divisible by $q$, so $y$ and $z$ must have the same residue modulo $q$.
Let $D_U$ be the number of pairs in $U$ whose two residues are different.
If $u_r$ is the number of elements of $U$ in residue class $r$, then
\[
 D_U=\frac12\left(m^2-\sum_r u_r^2\right).
\]
The class $r_0$ has size $m-\sigma$ and is at least as large as every other
class.  It follows that
\begin{equation}\label{eq:DU-piecewise}
 D_U\ge
 \begin{cases}
 \sigma(m-\sigma),&0\le\sigma\le m/2,\\[1mm]
 m\sigma/2,&m/2\le\sigma\le m.
 \end{cases}
\end{equation}
For the first line, the smallest possible number of unequal-residue pairs is
obtained by placing all $\sigma$ remaining elements in one other class.  For
the second line, every class has size at most $m-\sigma$, so
$\sum_r u_r^2\le(m-\sigma)m$.

We next count pairs $(x,y)$ with $x,y\in L$ that cannot be completed to a
copy.  List the elements of $L$ outside residue class $r_0$ as
$y_1<\cdots<y_\ell$, and let $j_i$ be the position of $y_i$ in the ordered
set $L$.  There are $j_i-1$ possible choices of $x<y_i$.  Once $x$ and
$y_i$ are chosen, the third point $z$ is forced.  Because $z-y_i=gd$, this
forced point must have the same residue as $y_i$.  Different choices of $x$
force different values of $z$.
Above $y_i$, there are at most $\ell-i$ elements of $L$ outside $r_0$ and
at most $\sigma$ such elements in $U$.  Hence at least
$j_i-1-(\ell-i+\sigma)$ of the earlier choices of $x$ cannot be completed,
whenever this number is positive.  If $D_L$ denotes the total number of
such impossible lower pairs, then $j_i\ge i$ gives
\begin{equation}\label{eq:DL-rank}
 \begin{aligned}
 D_L
 &\ge\sum_{i=1}^{\ell}
       \bigl(j_i-1-(\ell-i+\sigma)\bigr)_+\\
 &\ge\sum_{i=1}^{\ell}(2i-\ell-\sigma-1)_+
  =\left\lfloor\frac{(\ell-\sigma)_+^2}{4}\right\rfloor.
 \end{aligned}
\end{equation}
Here $(x)_+=\max\{x,0\}$.

It remains to express the two losses in terms of
$t=\sigma+\ell$, the total number of elements outside $r_0$.  Let
$f(\sigma)$ denote the right side of \eqref{eq:DU-piecewise}.  Since
$\ell=t-\sigma$, equations \eqref{eq:DU-piecewise} and
\eqref{eq:DL-rank} imply
\[
 D_U+D_L\ge f(\sigma)+\frac{(t-2\sigma)_+^2}{4}-1.
\]
We check the minimum of the right side.
If $t\le m$ and $\sigma\le t/2$, then $\sigma\le m/2$ and
\[
 f(\sigma)+\frac{(t-2\sigma)^2}{4}
 =\frac{t^2}{4}+\sigma(m-t)\ge\frac{t^2}{4}.
\]
If $t\le m$ and $t/2\le\sigma\le m/2$, then
\[
 f(\sigma)=\sigma(m-\sigma)
 \ge \frac t2\left(m-\frac t2\right)
 \ge\frac{t^2}{4}.
\]
If $\sigma\ge m/2$, then $f(\sigma)\ge m^2/4\ge t^2/4$.
Now write $t=m+u$ with $0\le u\le m$.  Since $\ell\le m$, we have
$\sigma\ge u$.  If $\sigma\le m/2$, then
\[
 f(\sigma)+\frac{(t-2\sigma)^2}{4}
 =\frac{t^2}{4}-u\sigma
 \ge\frac{m^2+u^2}{4}.
\]
If $m/2\le\sigma\le t/2$, subtracting $(m^2+u^2)/4$ leaves
\[
 \frac{(2\sigma-m)(\sigma-u)}2\ge0.
\]
Finally, if $\sigma\ge t/2$, then
\[
 f(\sigma)\ge\frac{mt}{4}
 \ge\frac{m^2+u^2}{4}.
\]
Therefore
\begin{equation}\label{eq:defect-t}
 D_U+D_L\ge
 \begin{cases}
 t^2/4-1,&0\le t\le m,\\[1mm]
 \bigl(m^2+(t-m)^2\bigr)/4-1,&m\le t\le2m.
 \end{cases}
\end{equation}
The residue class $r_0$ was chosen to be largest in $U$, not necessarily in
all of $S$.  Consequently $t\ge\nu_q(S)$.  The right side of
\eqref{eq:defect-t} increases with $t$, and at $t=\nu_q(S)$ it equals
$N^2G(\nu_q(S)/N)-1$.  Subtracting the impossible pairs from the
$m(m-1)$ available pairs gives
\[
 M_P^+(S)
 \le \frac14N^2-\frac N2
      -N^2G\left(\frac{\nu_q(S)}N\right)+1,
\]
which is stronger than the claimed estimate.

If $N$ is odd, add one integer $z$ in a globally largest residue class
modulo $q$, chosen far enough away that no copy involving $z$ is created.
Only finitely many values of $z$ are forbidden, so this is possible.  The
enlarged set has size $N+1$, the same value of $\nu_q$, and the same copy
count.  At a fixed integer $\nu$, the function $X^2G(\nu/X)$ does not
decrease as $X$ increases: it is constant for $X\ge2\nu$, while for
$\nu\le X\le2\nu$ it equals
\[
 \frac{X^2+(2\nu-X)^2}{16}.
\]
Apply the stronger even estimate above to the enlarged set.  Its first two
terms simplify as
\[
 \frac14(N+1)^2-\frac{N+1}{2}=\frac14N^2-\frac14.
\]
Together with the monotonicity just proved, this gives
\eqref{eq:residue-spread}.
\end{proof}

\subsection{Residue classes whose labels differ by a fixed amount}
Let $q=p^k$ be a prime power, where $p$ is prime and $k\ge1$, and let
$x_r\ge0$ be a weight attached to
each residue $r\pmod q$.  Write
\[
 |\mathbf x|=\sum_r x_r,
 \qquad
 \nu(\mathbf x)=|\mathbf x|-\max_r x_r.
\]
Thus $\nu(\mathbf x)$ is the total weight outside the largest class.
For $h\ne0\pmod q$, define
\begin{equation}\label{eq:def-fixed-difference-matching}
 S_h(\mathbf x)=\sum_r\min\{x_r,x_{r+h}\}.
\end{equation}
This quantity bounds how many objects can be matched between classes whose
residues differ by $h$.  Also let $E_h(\mathbf x)$ be the sum of $x_rx_s$
over each unordered pair of residue classes $\{r,s\}$ with $s=r+h$ or
$r=s+h$, counting each such pair once.  Set
\begin{equation}\label{eq:def-kappa-p}
 \kappa_p=
 \begin{cases}
 1/3,&p=3,\\
 1/4,&p\ne3.
 \end{cases}
\end{equation}
\begin{lemma}[Two bounds when residue labels differ by a fixed amount]
\label{lem:fixed-difference}
For every nonzero $h\pmod q$,
\begin{align}
 S_h(\mathbf x)
 &\le \min\left\{|\mathbf x|,\frac{p}{p-1}\nu(\mathbf x)\right\},
 \label{eq:fixed-difference-matching}\\
 E_h(\mathbf x)&\le\kappa_p|\mathbf x|^2.
 \label{eq:fixed-difference-pairs}
\end{align}
\end{lemma}
\begin{proof}
Repeatedly adding $h$ partitions the residue classes into cycles.  Let $L$ be the length of one of these cycles; equivalently, $L$ is the
smallest positive integer for which $Lh\equiv0\pmod q$.  Because $q$ is a
power of $p$, the number $L$ is also a power of $p$, and $L\ge p$.
Choose a cycle containing a largest weight and label its weights
$x_0,x_1,\ldots,x_{L-1}$, with $x_0$ largest.  The two terms involving
$x_0$ in $S_h$ are exactly $x_1+x_{L-1}$.  For $1\le i\le L-2$,
\[
 \min\{x_i,x_{i+1}\}
 \le \frac{i}{L-1}x_i+\frac{L-1-i}{L-1}x_{i+1}.
\]
After summing, every weight in this cycle other than $x_0$ has coefficient
$L/(L-1)$.  On every other cycle, the sum of the minima is at most the total
weight of that cycle.  Hence
\[
 S_h(\mathbf x)
 \le\frac{L}{L-1}\nu(\mathbf x)
 \le\frac{p}{p-1}\nu(\mathbf x).
\]
The simpler bound $S_h(\mathbf x)\le|\mathbf x|$ follows by summing
$\min\{x_r,x_{r+h}\}\le x_r$.

For the second estimate, we use the following elementary weighted-graph
fact.  If a graph never contains more than $\omega$ vertices that are all
joined to one another, then
\[
 \sum_{uv\in E}x_ux_v
 \le\frac{\omega-1}{2\omega}\left(\sum_vx_v\right)^2.
\]
To see this, if two vertices with positive weights are not joined, move their
combined weight to the one with the larger weighted sum of neighbors.  This
does not decrease the displayed edge sum.  Repeating leaves positive weight
only on vertices that are all joined to one another, where the inequality is
immediate.

The graph formed by the residue pairs above is a disjoint union of cycles.
Three residues can be mutually joined only when a cycle is a triangle, which
can occur only for $p=3$.  Thus we may take $\omega=3$ when $p=3$ and
$\omega=2$ otherwise, giving \eqref{eq:fixed-difference-pairs}.  If a cycle has
length two, it is a single unordered pair and is counted once in
$E_h$.
\end{proof}

\subsection{A bound for solutions drawn from three sets}\label{sec:three-set}
For $u,v,w\ge0$, sort the three numbers as $\rho\le\sigma\le\tau$ and
define
\begin{equation}\label{eq:def-Phi}
 \Phi(u,v,w)=
 \begin{cases}
 \rho\sigma,&\tau\ge\rho+\sigma,\\[1mm]
 \displaystyle
 \frac{2(uv+uw+vw)-(u^2+v^2+w^2)}4,&\tau<\rho+\sigma.
 \end{cases}
\end{equation}
The following form of Lev and Pinchasi's estimate works for every choice of
nonzero coefficients; see \cite[Lemma~2]{LevPinchasi} and compare
\cite[Lemma~2.6]{Aaronson}.
\begin{lemma}[Lev--Pinchasi bound for three prescribed sets]
\label{lem:three-set}
Let $X,Y,Z\subset\Z$ be finite, let $\alpha,\beta,\gamma$ be nonzero
integers, and let $\delta\in\Z$.  Then
\begin{equation}\label{eq:three-set}
 \#\{(x,y,z)\in X\times Y\times Z:
 \alpha x+\beta y+\gamma z=\delta\}
 \le \Phi(|X|,|Y|,|Z|)+\frac14.
\end{equation}
\end{lemma}
The estimate is usually stated for a homogeneous equation, with $\delta=0$.
Multiplying each of the three sets by a nonzero integer and translating one
set reduces the displayed equation to that form.  Multiplication by a
nonzero integer does not identify two distinct integers.

We will use two direct consequences of the definition of $\Phi$.  First,
with $(t)_+=\max\{t,0\}$,
\begin{equation}\label{eq:Phi-positive-part}
 \Phi(u,v,w)
 =uv-\frac14\pos{u+v-w}^{2}
  +\frac14\pos{u-v-w}^{2}
  +\frac14\pos{v-u-w}^{2}.
\end{equation}
Second, define
\begin{equation}\label{eq:def-psi}
 \psi(t)=1-\pos{1-t}^{2}
 =\begin{cases}2t-t^2,&0\le t\le1,\\1,&t\ge1.
 \end{cases}
\end{equation}
Then
\begin{equation}\label{eq:Phi-psi}
 \Phi(u,v,w)
 \le uv\,\psi\left(\frac{w}{u+v}\right)
 \qquad (u+v>0).
\end{equation}
If $w\ge u+v$, equality holds.  If the three numbers satisfy the triangle
inequalities, the claim reduces to $uv\le(u+v)^2/4$.  The two remaining
cases follow immediately from the first line of the definition of $\Phi$.
\section{Copies with Positive Dilation}\label{sec:positive-d}
\subsection{Splitting off the largest residue class}
After reflecting the pattern if necessary, we may write it as
\begin{equation}\label{eq:pattern-right-gap}
 P=\{0,a,a+g\},\qquad g>1.
\end{equation}
Indeed, if both adjacent gaps were $1$, the pattern would be
$\{0,1,2\}$.  Since the pattern is in coprime form, $\gcd(a,g)=1$.

Choose a prime $p\mid g$ and let
\begin{equation}\label{eq:one-sided-q}
 q=p^{v_p(g)}.
\end{equation}
Thus $a$ and $g/q$ are both coprime to $q$.
The next proposition separates the largest residue class modulo $q$ from
the rest of the set.  The notation $R+qB$ means that the elements in that
class are translated by $-R$ and divided by $q$, producing the smaller set
$B$.
\begin{proposition}[Counting after splitting off the largest residue class]
\label{prop:largest-class-split}
Let $A\subset\Z$ be finite and nonempty.  Choose a largest residue class
modulo $q$, write the part of $A$ in that class as $R+qB$, and put
$C=A\setminus(R+qB)$.  Set
\[
 u=|B|,\qquad v=|C|,\qquad
 s=\frac{\nu_q(B)}u.
\]
Then
\begin{equation}\label{eq:largest-class-split}
 M_P^+(A)
 \le M_P^+(B)+M_P^+(C)
 +uv\min\left\{\frac{p}{p-1}s,1\right\}
 +\kappa_pv^2.
\end{equation}
\end{proposition}
\begin{proof}
Write the three points of a copy as
\[
 x,\qquad Y=x+ad,\qquad Z=x+(a+g)d.
\]
Because $Z-Y=gd$ and $q\mid g$, the last two points satisfy
$Y\equiv Z\pmod q$.

There are four possibilities.  If all three points lie in $R+qB$, then
$q\mid d$, and translating and dividing by $q$ turns the copy into a copy
inside $B$.  If all three points lie outside $R+qB$, it is counted by
$M_P^+(C)$.

Next suppose that $x\in C$ while $Y,Z\in R+qB$.  Fix the residue class of
$x$ modulo $q$.  Since $a$ is coprime to $q$, the congruence
\[
 Y-x\equiv ad\pmod q
\]
determines $d$ modulo $q$, and this residue is nonzero.  After $Y$ and $Z$
are translated and divided by $q$, their residues modulo $q$ differ by some
fixed nonzero $h$.  If $B_r$ denotes the part of $B$ in residue class $r$,
then for each fixed $x$ the number of choices is at most
\[
 S_h\bigl((|B_r|)_r\bigr).
\]
Indeed, a copy gives a matched pair between the two prescribed residue
classes of $B$, and for fixed $x$ either later point determines the other.  By
\eqref{eq:fixed-difference-matching}, this is at most
\[
 \min\left\{u,\frac{p}{p-1}\nu_q(B)\right\}.
\]
There are $v$ possible choices of $x$, which gives the third term in
\eqref{eq:largest-class-split}.

Finally, suppose that $x\in R+qB$ and that $Y,Z$ lie in one fixed residue
class outside it.  The pair $(Y,Z)$ determines $x$.  Translate and divide
that outside class by $q$.  The allowed residue classes of $Y$ and $Z$ then
differ by a fixed nonzero amount, so \eqref{eq:fixed-difference-pairs} shows that a
class of size $v_i$ contributes at most $\kappa_pv_i^2$ copies.  Summing
over the outside classes and using $\sum_i v_i^2\le v^2$ gives the final
term.
\end{proof}
We now verify the numerical inequality that makes the induction work.  The
function below is exactly the difference between the desired coefficient
and the four terms in \eqref{eq:largest-class-split}, after all sizes are
divided by $u$.
\begin{lemma}[Numerical estimate for $99/400$]
\label{lem:numeric-99}
Put $\varepsilon=1/400$.  For a prime $p$, let
\[
 h_p(s)=\min\left\{\frac{p}{p-1}s,1\right\}.
\]
For $0\le r\le1/9$ and $0\le s\le1$, define
\begin{equation}\label{eq:def-Hp}
 H_p(r,s)=
 \max\{\varepsilon,G(s)\}-\varepsilon
 +\left(\frac12-2\varepsilon-h_p(s)\right)r
 -\kappa_pr^2.
\end{equation}
Then $H_p(r,s)\ge0$.
\end{lemma}
\begin{proof}
For fixed $s$, the function of $r$ is concave, so its minimum on
$0\le r\le1/9$ occurs at $r=0$ or $r=1/9$.  At $r=0$ the claim is
immediate.  If $p\ge5$, then $h_p(s)\le h_2(s)$ and
$\kappa_p=\kappa_2=1/4$, so the case $p=2$ gives the weaker bound.  It
therefore remains to check $p=2$ and $p=3$ at $r=1/9$.
For $p=2$, the formula changes at $s=1/10$, where
$G(s)=\varepsilon$, and at $s=1/2$, where $h_2(s)$ reaches $1$.
On $[0,1/10]$ the expression decreases, and at $s=1/10$ it equals
$481/16200$.  On $[1/10,1/2]$ it is the convex quadratic
\[
 \frac{s^2}{4}-\varepsilon
 +\frac19\left(\frac12-2\varepsilon-2s\right)-\frac1{324}.
\]
Its minimum occurs at $s=4/9$ and equals $1/32400$.  On
$[1/2,1]$, the function $h_2$ is constant and $G$ increases, so the
minimum is at $s=1/2$ and equals $13/16200$.
For $p=3$, the corresponding points are $s=1/10$, $1/2$, and $2/3$.
On $[0,1/10]$ the minimum is the endpoint value $1663/48600$.  On
$[1/10,1/2]$ the quadratic minimum occurs at $s=1/3$ and equals
$2003/97200$.  On $[1/2,2/3]$ the expression decreases, and on
$[2/3,1]$ it increases.  Its value at $s=2/3$ is
$653/97200$.  All these numbers are positive.
\end{proof}
\begin{proof}[Proof of \cref{thm:positive-general}]
Reflect the pattern if necessary so that it has the form
\eqref{eq:pattern-right-gap}, and choose $p$ and $q$ as above.  We prove the
stated finite estimate
\begin{equation}\label{eq:positive-induction}
 M_P^+(A)\le
 \left(\frac14-\frac1{400}\right)|A|^2+(|A|-1)
\end{equation}
by induction on $|A|$.
The case $|A|=1$ is immediate.  If all elements of $A$ lie in one residue
class modulo $q$, translate that class and divide it by $q$.  This leaves
the copy count unchanged.  Repeat until at least two residue classes are
used.  When $|A|>1$, this process terminates: two distinct integers cannot
remain congruent modulo arbitrarily high powers of $q$.

Put $n=|A|$.  If at least $n/10$ elements lie outside the largest residue
class, then $G(\nu_q(A)/n)\ge G(1/10)=1/400$.  Lemma
\ref{lem:residue-spread} gives
\[
 M_P^+(A)\le\left(\frac14-\frac1{400}\right)n^2+1,
\]
which implies \eqref{eq:positive-induction} because $n\ge2$.

We may therefore assume $\nu_q(A)<n/10$.  Apply
\cref{prop:largest-class-split}.  Let $u=|B|$, $v=|C|$,
$r=v/u$, and $s=\nu_q(B)/u$.  Since $v<n/10$ and $u=n-v$, we have
$0<v<u/9$.  The induction hypothesis gives one bound for $M_P^+(B)$, while
\cref{lem:residue-spread} gives another.  Taking the better of the two,
\[
 M_P^+(B)
 \le \frac14u^2-
 \max\left\{\frac1{400},G(s)\right\}u^2+u.
\]
For $C$, induction gives
\[
 M_P^+(C)
 \le\left(\frac14-\frac1{400}\right)v^2+(v-1).
\]
Substitute these two estimates into
\eqref{eq:largest-class-split}.  After moving the desired right side of
\eqref{eq:positive-induction} to the other side, the remaining quadratic
quantity is
\[
 u^2H_p(r,s),
\]
which is nonnegative by \cref{lem:numeric-99}.  The linear terms are
$u+(v-1)=n-1$.  This completes the induction.
\end{proof}

\section{Copies with Both Signs of the Dilation}
\label{sec:both-signs}
\subsection{Choosing an odd prime divisor}
The next observation explains why the arithmetic progression
$\{0,1,2\}$ is the only pattern excluded from the argument in this section.
\begin{lemma}\label{lem:odd-prime}
If $0<a<b$, $\gcd(a,b)=1$, and none of $a$, $b$, or $b-a$ has an odd
prime divisor, then $(a,b)=(1,2)$.
\end{lemma}
\begin{proof}
The three positive integers $a$, $b-a$, and $b$ must all be powers of $2$,
and
\[
 a+(b-a)=b.
\]
A sum of two powers of $2$ is itself a power of $2$ only when the two
summands are equal.  Hence $b=2a$.  Since $\gcd(a,b)=1$, we must have
$a=1$ and $b=2$.
\end{proof}

Fix a pattern $P\ne\{0,1,2\}$ in coprime form.  By
\cref{lem:odd-prime}, some odd prime $p$ divides one of its three pairwise
differences.  Relabel the three points as $\alpha,\beta,\gamma$ so that
\begin{equation}\label{eq:alpha-beta-gamma}
 \gamma-\alpha=q\ell,
 \qquad
 \beta-\alpha=m,
 \qquad
 q=p^{v_p(\gamma-\alpha)},
\end{equation}
where $p\nmid\ell m$.  The prime $p$ cannot divide two pairwise differences,
for then it would divide all of them and contradict the coprime form of the
pattern.  Thus both $\ell$ and $m$ have inverses modulo $q$.  Define
\begin{equation}\label{eq:def-lambda}
 \lambda=m\ell^{-1}\pmod q.
\end{equation}
For a copy of $P$, write
\[
 X=x+\alpha d,
 \qquad Y=x+\beta d,
 \qquad Z=x+\gamma d.
\]
These three integers satisfy
\begin{equation}\label{eq:pattern-relation}
 (\gamma-\beta)X+(\alpha-\gamma)Y+(\beta-\alpha)Z=0.
\end{equation}
All three coefficients are nonzero.  Consequently
\cref{lem:three-set} can be applied after we prescribe the residue classes
from which $X,Y,Z$ are chosen.

\subsection{Combining the two directions between residue classes}
Write $\Fq=\mathbb Z/q\mathbb Z$ for the set of residues modulo $q$, and
attach a nonnegative weight $x_r$ to each $r\in\Fq$.  For $h\ne0$, define
\begin{equation}\label{eq:def-Ph-Jh}
 P_h(\mathbf x)=\sum_{r\in\Fq}x_rx_{r+h},
 \qquad
 J_h(\mathbf x;w)=\sum_{r\in\Fq}\Phi(x_r,x_{r+h},w),
\end{equation}
and write $|\mathbf x|=\sum_r x_r$.  The quantity $P_h$ is the weighted
number of pairs in residue classes whose difference is $h$.  The quantity
$J_h$ is what the three-set bound gives after a third set of size $w$ is
added.

\begin{lemma}[A bound for adjacent residue classes]
\label{lem:adjacent-pair-bound}
For every nonnegative $\mathbf x$ and every $h\ne0$,
\begin{equation}\label{eq:adjacent-pair-bound}
 P_h(\mathbf x)\le \frac13|\mathbf x|^2.
\end{equation}
\end{lemma}
\begin{proof}
Because $q$ is odd, repeatedly adding $h$ gives cycles of odd length, and
$P_h(\mathbf x)$ counts each neighboring pair on those cycles once.
Therefore \eqref{eq:fixed-difference-pairs} gives
\[
 P_h(\mathbf x)\le\kappa_p|\mathbf x|^2\le\frac13|\mathbf x|^2.
\]
\end{proof}

The following one-variable calculation produces the coefficient $13/28$.
Its variables will later be the ratio of two class sizes and two pair counts
after division by the relevant squares.
\begin{lemma}[Numerical estimate for $13/28$]\label{lem:numeric-13-28}
For $0\le t\le1$ and $0\le s,u\le1/3$,
\begin{equation}\label{eq:numeric-13-28}
 s\,\psi\left(\frac{t}{2s}\right)+t^2u
 \le \frac67t+\frac1{28}(s+t^2u),
\end{equation}
where the first term is interpreted as $0$ when $s=0$.
\end{lemma}
\begin{proof}
Set $B=1/28$.  After subtracting $B(s+t^2u)$, it is enough to prove
\begin{equation}\label{eq:numeric-13-28-reduced}
 s\left(\psi\left(\frac{t}{2s}\right)-B\right)
 +(1-B)t^2u\le\frac67t.
\end{equation}
The case $t=0$ is immediate.  The left side increases with $u$, so we may
put $u=1/3$.

For fixed $t>0$, define
\[
 f_t(s)=s\left(\psi\left(\frac{t}{2s}\right)-B\right).
\]
If $s\le t/2$, then $\psi(t/(2s))=1$ and
$f_t(s)=(1-B)s$.  If $s\ge t/2$, then
\[
 f_t(s)=t-\frac{t^2}{4s}-Bs.
\]
The latter expression is largest at
$s=t/(2\sqrt B)=t\sqrt7$, unless this point lies beyond $1/3$.
Consequently
\[
 \max_{0\le s\le1/3}f_t(s)=
 \begin{cases}
 t(1-\sqrt B),&0<t\le1/(3\sqrt7),\\[1mm]
 \dfrac13\bigl(\psi(3t/2)-B\bigr),&t\ge1/(3\sqrt7).
 \end{cases}
\]
In the first range, divide the left side of
\eqref{eq:numeric-13-28-reduced} by $t$.  It is at most
\[
 1-\frac1{2\sqrt7}+\frac9{28}t
 \le1-\frac{11}{28\sqrt7}
 \le\frac67,
\]
where the last inequality is equivalent to $4\sqrt7\le11$.
In the second range, the largest value occurs at $s=u=1/3$.  For
$t\le2/3$, the left side of \eqref{eq:numeric-13-28} minus its right side is
\[
 -\frac{(6t-1)^2}{84}.
\]
For $2/3\le t\le1$, the same difference is
\[
 \frac{9t^2-24t+9}{28}\le0.
\]
This proves the lemma.
\end{proof}
\begin{proposition}[Combining the two possible directions]
\label{prop:opposite-differences}
Let $q$ be odd, let $h\ne0$ in $\Fq$, and suppose
$|\mathbf x|=a$ and $|\mathbf y|=b$.  Then
\begin{equation}\label{eq:opposite-differences}
 J_h(\mathbf x;b)+J_{-h}(\mathbf y;a)
 \le \frac67ab
 +\frac1{28}\bigl(P_h(\mathbf x)+P_{-h}(\mathbf y)\bigr).
\end{equation}
\end{proposition}
\begin{proof}
Assume $0\le b\le a$; the other case follows by interchanging the two
vectors.  If $b=0$, both terms on the left vanish, so the claim is immediate.

For every $r$ we have $a\ge b\ge y_r+y_{r-h}$, and therefore
\begin{equation}\label{eq:large-third-set}
 J_{-h}(\mathbf y;a)=P_{-h}(\mathbf y).
\end{equation}
If $P_h(\mathbf x)=0$, then $J_h(\mathbf x;b)=0$, while
\cref{lem:adjacent-pair-bound} gives
$P_{-h}(\mathbf y)\le b^2/3\le ab/3$; hence the claim is immediate.

We may therefore assume $P_h(\mathbf x)>0$.  The function $\psi$ is concave
and nondecreasing.  Apply \eqref{eq:Phi-psi} term by term and average with
weights $x_rx_{r+h}$.  This gives
\begin{align}
 J_h(\mathbf x;b)
 &\le P_h(\mathbf x)\,
 \psi\left(
  \frac{b}{P_h(\mathbf x)}
  \sum_r\frac{x_rx_{r+h}}{x_r+x_{r+h}}
 \right)\notag\\
 &\le P_h(\mathbf x)\,
 \psi\left(\frac{ab}{2P_h(\mathbf x)}\right).
 \label{eq:cycle-Jensen}
\end{align}
The last step uses
\[
 \frac{x_rx_{r+h}}{x_r+x_{r+h}}
 \le\frac{x_r+x_{r+h}}4
\]
and hence
\[
 \sum_r\frac{x_rx_{r+h}}{x_r+x_{r+h}}\le\frac a2.
\]
Terms with zero denominator may simply be omitted.

Set
\[
 t=\frac ba,
 \qquad
 s=\frac{P_h(\mathbf x)}{a^2},
 \qquad
 u=\frac{P_{-h}(\mathbf y)}{b^2}.
\]
By \cref{lem:adjacent-pair-bound}, $s,u\le1/3$.  Dividing the desired
inequality by $a^2$ now turns it exactly into \cref{lem:numeric-13-28}.
\end{proof}

\subsection{Repeatedly splitting into residue classes}

\begin{proof}[Proof of \cref{thm:both-signs-general}]
Choose an integer $T$ large enough that distinct elements of $A$ are not
congruent modulo $q^T$.  At level $t$, split $A$ into its nonempty residue
classes modulo $q^t$.  Join each class to the finer classes modulo
$q^{t+1}$ that it contains.  This produces a rooted tree whose leaves are
the individual elements of $A$.  The tree is simply a convenient way to record the first level at which
points become separated.

Consider a nonconstant copy and write $d=q^tu$ with $q\nmid u$.  At the
node representing the common residue class modulo $q^t$, the points $X$ and
$Z$ lie in one child, while $Y$ lies in a different child.  At the next
level, $X$ and $Z$ lie in different grandchildren.  This is the unique node
to which we assign the copy.

To see the residue relation explicitly, let $i$ be the child residue of
$X$ and $Z$, let $j$ be the child residue of $Y$, and let $r,s$ be the
grandchild residues of $X,Z$ within child $i$.  Equations
\eqref{eq:alpha-beta-gamma} give
\begin{equation}\label{eq:general-digit-rule}
 j-i\equiv mu,
 \qquad
 s-r\equiv\ell u,
 \qquad
 j-i\equiv\lambda(s-r)\pmod q.
\end{equation}

Fix a node $v$.  Let $a_i$ be the number of elements in child $i$, and put
\[
 N_v=\sum_i a_i.
\]
Within child $i$, let
\[
 \mathbf b_i=(b_{i,r})_{r\in\Fq},
 \qquad |\mathbf b_i|=a_i,
\]
where $b_{i,r}$ is the size of grandchild $r$.  Let $C_v$ be the number of
copies assigned to $v$.  After fixing $i$, $j$, and the two grandchild
residues allowed by \eqref{eq:general-digit-rule}, apply
\cref{lem:three-set}.  Summing those bounds gives
\begin{equation}\label{eq:copies-at-node}
 C_v
 \le
 \sum_{i\ne j}
 J_{\lambda^{-1}(j-i)}(\mathbf b_i;a_j)
 +\frac{q^2(q-1)}4.
\end{equation}
There are $q$ choices of $i$, $q-1$ choices of the nonzero difference
$j-i$, and $q$ choices of the first grandchild residue.  This explains the
$q^2(q-1)$ additive terms of size $1/4$.

Combine the terms for $(i,j)$ and $(j,i)$ using
\cref{prop:opposite-differences}.  Define
\begin{align}
 D_v&=N_v^2-\sum_i a_i^2,
 \label{eq:def-general-D}\\
 E_v&=\sum_i\left(a_i^2-\sum_r b_{i,r}^2\right).
 \label{eq:def-general-E}
\end{align}
The number $D_v$ is the number of ordered pairs of elements that lie in
different children of $v$.  The number $E_v$ is the number of ordered pairs
that lie in the same child of $v$ but in different grandchildren.

For fixed $i$, as $j\ne i$ varies, the values
$\lambda^{-1}(j-i)$ run through all nonzero residues.  Therefore
\[
 \sum_{h\ne0}P_h(\mathbf b_i)
 =a_i^2-\sum_r b_{i,r}^2.
\]
It follows that every node with at least two nonempty children satisfies
\begin{equation}\label{eq:node-pair-bound}
 C_v
 \le \frac37D_v+\frac1{28}E_v+\frac{q^2(q-1)}4.
\end{equation}
If a node has only one nonempty child, no copy is assigned there and no
additive term is needed.

Every ordered pair of distinct elements of $A$ first lies in different
children at exactly one node.  Hence
\begin{equation}\label{eq:sum-general-D}
 \sum_vD_v=n^2-n.
\end{equation}
Also, $E_v$ is the sum of $D_w$ over the children $w$ of $v$.  Thus
\begin{equation}\label{eq:sum-general-E}
 \sum_vE_v=\sum_{w\ne\mathrm{root}}D_w\le n^2-n.
\end{equation}

A rooted tree with $n$ leaves has at most $n-1$ nodes with at least two
children.  Summing \eqref{eq:node-pair-bound} over all nodes gives
\[
 M_P(A)
 \le\left(\frac37+\frac1{28}\right)(n^2-n)
 +\frac{q^2(q-1)}4(n-1),
\]
which is exactly \cref{thm:both-signs-general}.
\end{proof}

\section{A Sharper Bound for \texorpdfstring{$\{0,1,3\}$}{0,1,3}}
\label{sec:013-stronger}
We now specialize to
\begin{equation}\label{eq:main-equation}
 2x+z=3y.
\end{equation}
Use the repeated splitting from the previous section with $q=3$.  Assign a
nonconstant solution to the first node at which its three points separate.
At that node, $x$ and $z$ lie in one child, $y$ lies in another child, and
one level lower $x$ and $z$ lie in different grandchildren.  If $i$ is the
child residue of $x,z$, if $j$ is the child residue of $y$, and if $r,s$
are the grandchild residues of $x,z$, then reducing the equation modulo $3$
gives
\begin{equation}\label{eq:mod3-digit-rule}
 j\equiv i+s-r\pmod3.
\end{equation}
The improvement over \cref{thm:both-signs-general} comes from remembering
the three separate class sizes at each split instead of keeping only their
total.

Throughout this section, set
\begin{equation}\label{eq:c-eta-rho}
 c=\frac{47}{122},
 \qquad
 \eta=\frac{19}{732}=\frac{c-1/3}{2},
 \qquad
 \rho=\frac{1-c}{2}=\frac13-\eta=\frac{75}{244}.
\end{equation}
\subsection{The loss from an uneven split into three residue classes}
For three nonnegative numbers $\mathbf x=(x_0,x_1,x_2)$, define
\begin{equation}\label{eq:def-three-class-J}
 J(\mathbf x;w)=\sum_{0\le i<j\le2}\Phi(x_i,x_j,w).
\end{equation}
This adds the three possible bounds obtained by choosing two of the three
residue classes for two variables and using a third set of size $w$.
Now let $\mathbf q=(q_0,q_1,q_2)$ consist of three nonnegative numbers with sum
$1$.  Define
\begin{equation}\label{eq:def-F-delta-beta}
 F(w)=J((1/3,1/3,1/3);w)
 =
 \begin{cases}
 w-\dfrac34w^2,&0\le w\le2/3,\\[1mm]
 1/3,&w\ge2/3,
 \end{cases}
\end{equation}
\begin{equation}\label{eq:def-delta}
 \delta(\mathbf q)=\frac12\left(\sum_{i=0}^2q_i^2-\frac13\right),
\end{equation}
and
\begin{equation}\label{eq:def-beta}
 \beta(w)=
 \begin{cases}
 9w-14w^2,&0\le w\le3/8,\\[1mm]
 1,&w>3/8.
 \end{cases}
\end{equation}
The downward jump in $\beta$ at $w=3/8$ is intentional: at the endpoint the
stronger quadratic loss remains available, whereas for $w>3/8$ the argument
uses only the uniform factor $1$.

The number $\delta(\mathbf q)$ measures how unequal the three parts are.  It is
zero exactly when $q_0=q_1=q_2=1/3$.
\begin{lemma}[Loss caused by an uneven three-way split]
\label{lem:uneven-three-split}
For every such $\mathbf q$ and every $w\ge0$,
\begin{equation}\label{eq:uneven-three-split}
 F(w)-J(\mathbf q;w)
 \ge \beta(w)\min\{\delta(\mathbf q),\eta\}.
\end{equation}
\end{lemma}
\begin{proof}
First suppose $\delta(\mathbf q)\le\eta$, and reorder the three parts so that
$q_0\ge q_1\ge q_2$.  Put
\[
 A=q_0-q_1,
 \qquad B=q_1-q_2,
 \qquad s=A+B=q_0-q_2.
\]
A direct calculation gives
\begin{equation}\label{eq:delta-AB}
 \delta(\mathbf q)=\frac{A^2+AB+B^2}{3},
 \qquad
 s^2\le4\delta(\mathbf q)\le4\eta<\frac19.
\end{equation}
Thus $s<1/3$.  Also $q_1,q_2\ge q_0-s$, so
$q_0\le(1+2s)/3<5/9$.  It follows that every sum $q_i+q_j$ is greater
than $4/9$.

We divide the proof according to the size of $w$.  First suppose
$0\le w<s$, and put $z=w/s$.  Equation
\eqref{eq:Phi-positive-part} gives the exact identity
\begin{equation}\label{eq:small-w-exact}
 F(w)-J(\mathbf q;w)
 =\frac32\delta(\mathbf q)
 -\frac14\left((A-w)_+^2+(B-w)_+^2+(s-w)^2\right).
\end{equation}
Since
\[
 (A-w)_+\le A(1-z),
 \qquad
 (B-w)_+\le B(1-z),
\]
and $A^2+B^2+s^2=6\delta(\mathbf q)$, we obtain
\[
 F(w)-J(\mathbf q;w)
 \ge\delta(\mathbf q)\left(3z-\frac32z^2\right).
\]
Because $s<1/3$, we have $z=w/s\ge3w$.  The function
$3z-3z^2/2$ increases on $[0,1]$, and therefore
\[
 F(w)-J(\mathbf q;w)
 \ge\delta(\mathbf q)\left(9w-\frac{27}{2}w^2\right)
 \ge\delta(\mathbf q)(9w-14w^2).
\]
Next suppose $s\le w\le3/8$.  Each triple $(q_i,q_j,w)$ satisfies the
triangle inequalities, so \eqref{eq:Phi-positive-part} simplifies to
\begin{equation}\label{eq:middle-w-three}
 F(w)-J(\mathbf q;w)=\frac32\delta(\mathbf q).
\end{equation}
This is enough because $9w-14w^2\le81/56<3/2$ on this interval.
For $3/8<w\le2/3$, write $q_i=1/3+r_i$ and put $v=2/3-w$.
Another use of \eqref{eq:Phi-positive-part} gives
\begin{equation}\label{eq:large-w-three}
 F(w)-J(\mathbf q;w)
 =\delta(\mathbf q)-\frac34v^2
 +\frac14\sum_{i=0}^2(v-r_i)_+^2.
\end{equation}
The function $x\mapsto(v-x)_+^2$ is convex and $\sum_i r_i=0$, so the
sum is at least $3v^2$.  Hence
$F(w)-J(\mathbf q;w)\ge\delta(\mathbf q)$.  The same conclusion holds for $w\ge2/3$,
because
\[
 J(\mathbf q;w)\le\sum_{i<j}q_iq_j
 =\frac13-\delta(\mathbf q)=F(w)-\delta(\mathbf q).
\]
Since $\beta(w)=1$ for $w>3/8$, this proves the lemma when
$\delta(\mathbf q)\le\eta$.

It remains to remove that assumption.  Moving two coordinates closer
together while keeping their sum fixed cannot decrease $J(\mathbf q;w)$.  For the
terms involving the third coordinate, this follows from the concavity of
$t\mapsto\Phi(t,a,w)$.  Indeed, this function is linear for
$t\le |a-w|$, quadratic with second derivative $-1/2$ for
$|a-w|\le t\le a+w$, and constant for $t\ge a+w$, with nonincreasing
one-sided slopes at the junctions.  For the term involving the two
coordinates being balanced, write them as $m+t$ and $m-t$; then
\[
 \Phi(m+t,m-t,w)-\Phi(m,m,w)
 =-t^2+\frac14(2|t|-w)_+^2\le0.
\]
If $\delta(\mathbf q)>\eta$, move $\mathbf q$ toward the equal split by setting
\[
 \mathbf q'=\left(\frac13,\frac13,\frac13\right)
 +\sqrt{\frac{\eta}{\delta(\mathbf q)}}
 \left(\mathbf q-\left(\frac13,\frac13,\frac13\right)\right).
\]
Write $\theta=\sqrt{\eta/\delta(\mathbf q)}$.  If $q_i>1/3$, then
$q_i'=1/3+\theta(q_i-1/3)$ lies between $1/3$ and $q_i$; if $q_i<1/3$,
then $q_i'$ lies between $q_i$ and $1/3$.  Transfer the required mass from
coordinates above $1/3$ to coordinates below $1/3$, stopping whenever one
coordinate reaches its target value.  Each transfer moves the chosen pair
closer together without passing equality, and at most two transfers are
needed.  Thus this movement is a composition of the balancing operations
just described, so $J(\mathbf q;w)\le J(\mathbf q';w)$.  Since $\delta(\mathbf q')=\eta$, applying
the case already proved to $\mathbf q'$ gives \eqref{eq:uneven-three-split}.
\end{proof}

We need to use the preceding loss at two consecutive levels of the
modulo-$3$ splitting.  To carry the information from one level to the next, define the nonnegative
correction
\begin{equation}\label{eq:def-V}
 V(\mathbf q)=\left(c-\sum_{i=0}^2q_i^2\right)_{+}
 =2\pos{\eta-\delta(\mathbf q)}.
\end{equation}
It is largest near an equal split and vanishes once the three parts are
sufficiently unequal.  For two possible sizes $u,v$ of the third set,
define
\begin{equation}\label{eq:def-K}
 K(u,v)=F(u)+F(v)-\eta\min\{2,\beta(u)+\beta(v)\}.
\end{equation}
\begin{lemma}[Combining two target sizes]\label{lem:two-target-sizes}
For every three-part distribution $\mathbf q$ and all $u,v\ge0$,
\begin{equation}\label{eq:two-target-sizes}
 J(\mathbf q;u)+J(\mathbf q;v)-V(\mathbf q)\le K(u,v).
\end{equation}
\end{lemma}
\begin{proof}
Let $B=\beta(u)+\beta(v)$ and $\delta=\delta(\mathbf q)$.  Applying
\cref{lem:uneven-three-split} twice gives
\[
 J(\mathbf q;u)+J(\mathbf q;v)
 \le F(u)+F(v)-B\min\{\delta,\eta\}.
\]
After subtracting $V(\mathbf q)=2(\eta-\delta)_+$, the largest possible value of
the remaining $\delta$-dependent expression is
$-\eta\min\{2,B\}$.  This is precisely the definition of $K(u,v)$.
\end{proof}

\subsection{The inequality at one three-way split}
The next lemma is the algebraic heart of the sharper bound.  Suppose a node
is divided into three children whose relative sizes are $p_0,p_1,p_2$.
The left side below combines the correction assigned to the node with the
bounds for copies that use each child.  The right side is $c$ times the
fraction of ordered pairs that lie in different children.

\begin{lemma}[The inequality needed at one split]
\label{lem:one-three-way-split}
For every $\mathbf p=(p_0,p_1,p_2)$ with $p_i\ge0$ and $\sum_i p_i=1$,
\begin{equation}\label{eq:one-three-way-split}
 V(\mathbf p)+\sum_{i=0}^2p_i^2
 K\left(\frac{p_j}{p_i},\frac{p_k}{p_i}\right)
 \le c\left(1-\sum_{i=0}^2p_i^2\right),
\end{equation}
where $\{i,j,k\}=\{0,1,2\}$.  If $p_i=0$, the corresponding term is
understood as its limiting value.
\end{lemma}
\begin{proof}
The proof is an explicit quadratic calculation.  We first sort the three
parts and express their ratios using two variables.  The maximum with zero in $V$ then gives two cases.  One case follows from a simple uniform bound.  In
the other, the relevant functions are quadratic on three intervals; after
clearing denominators, only four expressions remain to be checked.
Set
\begin{equation}\label{eq:def-H-g}
 H(t)=F(t)-\eta\beta(t),
 \qquad
 g(t)=\beta(t)-1.
\end{equation}
Then
\begin{equation}\label{eq:K-Hg}
 K(u,v)=H(u)+H(v)+\eta(g(u)+g(v))_+.
\end{equation}
For $t\ge1$, we have $F(t)=1/3$, $\beta(t)=1$, and hence
$H(t)=\rho$ and $g(t)=0$.  For every $t\ge0$,
\begin{equation}\label{eq:single-rho}
 H(t)+\eta(g(t))_+
 =F(t)-\eta\min\{\beta(t),1\}\le\rho.
\end{equation}
Indeed, if $\beta(t)\ge1$, use $F(t)\le1/3$.  If $\beta(t)<1$, then
$t<1/7$, and the left side is at most
$F(1/7)=25/196<\rho$.

By symmetry, assume $p_0\ge p_1\ge p_2$.  The cases with a zero coordinate
are obtained by taking $r\downarrow0$ when $p_2=0<p_1$, or
$x\downarrow0$ when $p_1=p_2=0$, in the parameterization below.  Terms with
the corresponding squared prefactors vanish, and the remaining terms have
finite one-sided limits.  It therefore suffices to write
\begin{equation}\label{eq:param-p}
 \mathbf p=\frac{(1,x,xr)}d,
 \qquad d=1+x+xr,
 \qquad 0<x,r\le1,
\end{equation}
and put
\begin{equation}\label{eq:def-Q}
 Q=x+xr+x^2r.
\end{equation}
Thus $2Q/d^2=1-\sum_i p_i^2$.  After multiplying
\eqref{eq:one-three-way-split} by $d^2$, the right side minus the left side is
\begin{equation}\label{eq:def-Gamma}
 \Gamma=2cQ-2(Q-\rho d^2)_+
 -K(x,xr)-x^2K(1/x,r)-2\rho x^2r^2.
\end{equation}
We prove $\Gamma\ge0$.

First suppose $Q\ge\rho d^2$.  Equations \eqref{eq:K-Hg} and
\eqref{eq:single-rho}, together with
$(a+b)_+\le a_++b_+$, give
\[
 K(x,xr)\le2\rho,
 \qquad
 K(1/x,r)\le2\rho.
\]
In this case \eqref{eq:def-Gamma} simplifies to
\[
 \Gamma=(2\rho-K(x,xr))
 +x^2(2\rho-K(1/x,r))\ge0.
\]
Now suppose $Q\le\rho d^2$.  Define two auxiliary functions
\begin{equation}\label{eq:def-RS}
 R(t)=732\bigl(2ct-\rho t^2-H(t)\bigr),
 \qquad
 S(t)=R(t)-19g(t).
\end{equation}
The factor $732$ clears the denominators, and the two functions record the
two possibilities created by the positive part in \eqref{eq:K-Hg}: $R$ is
the inactive case, while $S=R-19g$ is the active case.  Since
$732\eta=19$, direct substitution gives
\begin{equation}\label{eq:min-RS}
 732\Gamma
 =\min\{R(x)+R(xr),S(x)+S(xr)\}
 +x^2\min\{R(r),S(r)\}.
\end{equation}
Introduce the three quadratic polynomials
\begin{equation}\label{eq:def-ABC}
 A(t)=3t+58t^2,
 \qquad
 B(t)=(6t-1)(54t-19),
 \qquad
 C(t)=-3(75t^2-188t+75).
\end{equation}
Substituting the definitions of $F$ and $\beta$ yields
\begin{equation}\label{eq:R-pieces}
 R(t)=
 \begin{cases}
 A(t),&0\le t\le3/8,\\
 B(t),&3/8<t\le2/3,\\
 C(t),&2/3<t\le1,
 \end{cases}
 \qquad
 S(t)=
 \begin{cases}
 B(t),&0\le t\le2/3,\\
 C(t),&2/3<t\le1.
 \end{cases}
\end{equation}
These formulas imply
\begin{equation}\label{eq:RS-basic}
 R(t)\ge0,
 \qquad
 R(t)\ge\frac{25}{9}t^2,
 \qquad
 S(t)\ge-\frac{25}{9}
 \qquad(0\le t\le1).
\end{equation}
For the last bound, note that
\begin{equation}\label{eq:B-square}
 B(t)=324\left(t-\frac7{27}\right)^2-\frac{25}{9}.
\end{equation}
Also $S(t)=C(t)\ge51$ when $t>2/3$.

Expanding the two minima in \eqref{eq:min-RS} gives four possible sums.
Two follow immediately from \eqref{eq:RS-basic}:
\[
 R(x)+R(xr)+x^2R(r)\ge0
\]
and
\[
 R(x)+R(xr)+x^2S(r)
 \ge R(x)-\frac{25}{9}x^2\ge0.
\]
For the third sum, we prove
\begin{equation}\label{eq:SS-goal}
 S(x)+S(xr)+x^2S(r)\ge0.
\end{equation}
If $x>2/3$, the bound $S(x)\ge51$ and
$S\ge-25/9$ make this immediate.  Suppose $x\le2/3$, so also
$xr\le2/3$.  If $r\le2/3$, all three terms use the polynomial $B$, and
\begin{equation}\label{eq:critical-sos}
 \begin{aligned}
 B(x)+B(xr)+x^2B(r)
 ={}&648\left(xr-\frac{7(x+1)}{54}\right)^2\\
 &+\frac{61}{9}(7x-2)^2\ge0.
 \end{aligned}
\end{equation}
If $r>2/3$, put $a(r)=99r^2+564r+99$.  Then
\begin{equation}\label{eq:SS-large-r}
 \begin{aligned}
 B(x)+B(xr)+x^2C(r)
 ={}&a(r)\left(x-\frac{84(1+r)}{a(r)}\right)^2\\
 &+\frac{366(20r-9r^2-9)}{a(r)}\ge0.
 \end{aligned}
\end{equation}
The final numerator is positive on $[2/3,1]$.

The remaining sum is
\begin{equation}\label{eq:SR-goal}
 S(x)+S(xr)+x^2R(r).
\end{equation}
If $r\ge1/7$, then $g(r)\ge0$, so $R(r)\ge S(r)$, and
\eqref{eq:SS-goal} proves the claim.  Suppose $r<1/7$.  If $x>2/3$,
then $S(x)\ge51$, while $S(xr)>0$ and $R(r)\ge0$.  If $x\le2/3$, put
$d(r)=382r^2+3r+324$.  Here $S(x)=S(xr)=B$ and $R(r)=A$, and
\begin{equation}\label{eq:SR-small-r}
 \begin{aligned}
 B(x)+B(xr)+x^2A(r)
 ={}&d(r)\left(x-\frac{84(1+r)}{d(r)}\right)^2\\
 &+\frac{2(3730r^2-6999r+2628)}{d(r)}\ge0.
 \end{aligned}
\end{equation}
The final quadratic decreases and remains positive on $[0,1/7]$.
Thus all four possible sums in \eqref{eq:min-RS} are nonnegative, proving
\eqref{eq:one-three-way-split}.
\end{proof}
The coefficient $47/122$ is chosen so that the most difficult expression
above becomes the sum of squares in \eqref{eq:critical-sos}.  Equality in
that identity away from the equal split occurs at
\begin{equation}\label{eq:equality-distribution}
 x=\frac27,
 \qquad
 r=\frac7{12},
 \qquad
 \mathbf p=\frac1{61}(42,12,7).
\end{equation}
Equivalently, choosing $\eta=19/732$ and
$c=1/3+2\eta=47/122$ produces exactly the two squares in
\eqref{eq:critical-sos}.

\subsection{Summing over all three-way splits}

\begin{proof}[Proof of \cref{thm:47-122}]
Return to the repeated splitting modulo $3$.  At a node $v$, let $N_v$ be
the number of elements at that node, let
$\mathbf p=(p_0,p_1,p_2)$ be the fractions in its three children, and let
$\mathbf q_i$ be
the three fractions among the grandchildren inside child $i$.  Let $C_v$
be the number of solutions assigned to $v$.  The congruence
\eqref{eq:mod3-digit-rule} and \cref{lem:three-set} give
\begin{equation}\label{eq:node-count-relative}
 C_v
 \le N_v^2\sum_{i=0}^2p_i^2
 \left[
 J\left(\mathbf q_i;\frac{p_j}{p_i}\right)
 +J\left(\mathbf q_i;\frac{p_k}{p_i}\right)
 \right]+\frac92,
\end{equation}
where $\{i,j,k\}=\{0,1,2\}$.  There are $18$ prescribed choices of
residue classes, and each application of \cref{lem:three-set} contributes
at most $1/4$, which explains the additive term $18/4=9/2$.

Assign to the node the correction
\begin{equation}\label{eq:def-U-D}
 U_v=N_v^2V(\mathbf p),
 \qquad
 D_v=N_v^2\left(1-\sum_i p_i^2\right).
\end{equation}
The number $D_v$ is the number of ordered pairs of elements in different
children.  The correction assigned to child $i$ is
$(N_vp_i)^2V(\mathbf q_i)$.  Applying \cref{lem:two-target-sizes} inside each child
and then \cref{lem:one-three-way-split} to the three child sizes gives
\begin{equation}\label{eq:node-correction}
 C_v+U_v-\sum_iU_{vi}
 \le \frac{47}{122}D_v+\frac92
\end{equation}
at every node with at least two nonempty children.  At a node with only one
nonempty child, the same inequality holds without the additive term.

Now sum this inequality over the whole tree.  Every correction at an
internal node appears once with a plus sign at that node and once with a
minus sign at its parent, so these terms cancel.  The corrections at leaves
are zero, and the correction at the root is nonnegative.  As before, every
ordered pair of distinct elements first separates at exactly one node, so
\[
 \sum_vD_v=n^2-n.
\]
There are at most $n-1$ nodes with at least two nonempty children.  Hence
\[
 M(A)+U_{\mathrm{root}}
 \le \frac{47}{122}(n^2-n)+\frac92(n-1).
\]
Dropping the nonnegative root correction proves \cref{thm:47-122}.
\end{proof}

\section*{Acknowledgements}

The author supplied the underlying residue-splitting idea and provided
high-level guidance on the direction of the work. GPT-5.6 Pro
was used extensively to generate and refine the detailed proofs, carry
out the technical calculations, and assist with drafting and revision.
The author independently checked the final arguments and calculations
and assumes full responsibility for all mathematical claims and any
remaining errors.

\end{document}